\documentclass[12pt,english]{extarticle}
\usepackage[utf8]{inputenc}
\usepackage[a4paper]{geometry}
\usepackage{fancyhdr}
\usepackage{color}
\usepackage{verbatim}
\usepackage{amsmath}
\usepackage{amsthm}
\usepackage{amssymb}

\makeatletter
\usepackage{xcolor}
\theoremstyle{plain}
\newtheorem{thm}{\protect\theoremname}[section]
\theoremstyle{plain}
\newtheorem{lem}[thm]{\protect\lemmaname}
\theoremstyle{plain}
\newtheorem{prop}[thm]{\protect\propositionname}

\usepackage[english]{babel}
\usepackage{amsthm}
\usepackage{amsfonts}

\date{}

\usepackage{color}

\definecolor{rot}{rgb}{1.000,0.000,0.000}

\newcommand{\N}{\mathbb{N}}
\newcommand{\R}{\mathbb{R}}

\newcommand{\e}{\mathrm{e}}

\newcommand{\dom}{\mathrm{dom}}

\renewcommand{\d}{\mathrm{d}}

\newcommand{\cA}{\mathcal{A}}
\newcommand{\cB}{\mathcal{B}}
\newcommand{\cC}{\mathcal{C}}
\newcommand{\cT}{\mathcal{T}}
\newcommand{\cS}{\mathcal{S}}
\newcommand{\cI}{\mathcal{I}}
\newcommand{\oA}{\mathbf{A}}
\newcommand{\oB}{\mathbf{B}}
\newcommand{\oI}{\mathbf{I}}
\newcommand{\oX}{\mathbf{X}}
\newcommand{\oC}{\mathbf{C}}
\newcommand{\oD}{\mathbf{D}}
\newcommand{\oE}{\mathbf{E}}
\newcommand{\oF}{\mathbf{F}}
\newcommand{\oT}{\mathbf{T}}

\definecolor{rot}{rgb}{1.000,0.000,0.000}

\makeatother

\usepackage{babel}
\providecommand{\lemmaname}{Lemma}
\providecommand{\propositionname}{Proposition}
\providecommand{\theoremname}{Theorem}

\begin{document}
\title{Trotter-type formula for operator semigroups\\ on product spaces}
\author{Artur Stephan\thanks{Weierstraß-Institut für Angewandte Analysis und Stochastik, Mohrenstraße 39, 10117 Berlin, Germany, e-mail: \tt{artur.stephan@wias-berlin.de}}}
\date{June 30, 2023}
\maketitle

\lhead{Trotter on product spaces}

\chead{June 30, 2023}

\rhead{Artur Stephan}
\begin{abstract}
We consider a Trotter-type-product formula for approximating the solution
of a linear abstract Cauchy problem (given by a strongly continuous semigroup), where the underlying Banach space is a product of two
spaces. In contrast to the classical Trotter-product formula, the
approximation is given by freezing subsequently the components of
each subspace. After deriving necessary stability estimates for the
approximation, which immediately provide convergence in the natural
strong topology, we investigate convergence in the operator norm.
The main result shows that an almost optimal convergence rate can be
established if the dominant operator generates a holomorphic semigroup
and the off-diagonal coupling operators are bounded.

\end{abstract}

\section{Introduction}

The classical and rich theory of strongly continuous semigroups $\left\{ \oT(t)\right\} {}_{t\geq0}$
provides a tool for solving linear abstract Cauchy problems $\dot{u}(t)=-\oC u(t)$
for $t\geq0$, $u(0)=x$ in a Banach space $X$ \cite{Pazy83SLOA,Kato95PTLO,EngNag00OPSLEE}:
the linear operator $-\oC:\dom(\oC)\subset X\to X$ is a generator
of a strongly continuous semigroup (denoted by $\e^{-t\oC}$) if and
only if for every $x\in\dom(\oC)$ there exists a unique solution
$u(\cdot,x)\in\dom(\oC)$ of the abstract Cauchy problem, which is
given by $\e^{-t\oC}x=u(t,x)$ (we recall basics from semigroup theory
in Section \ref{subsec:Semigroup-theory}). Often the operator $\oC$
is given by a sum $\oC=\oA+\oB$. If $-\oA$ is a generator, classical
perturbation results (see e.g. \cite[Chapter III]{EngNag00OPSLEE})
provide information when $-\oC$ is again a generator. Moreover, assuming
that $-\oA$ and $-\oB$ are generators of semigroups $\e^{-t\oA}$
and $\e^{-t\oB}$, respectively, then the solution operator $\e^{-t\oC}$
can be approximated by the so-called Trotter-product formula
\[
\lim_{n\to\infty}\left(\e^{-\tfrac{t}{n}\oA}\e^{-\tfrac{t}{n}\oB}\right)^{n}\to\e^{-t\oC},
\]
provided a stability condition holds for the product (see e.g. \cite[Corollary III.5.8]{EngNag00OPSLEE}).
The convergence is an immediate consequence of the more general Chernoff-product
formula \cite{Cher74PFNSAUO} and is meant in the natural topology
in the theory of semigroups, the strong topology. Apart from its theoretical
value, the Trotter-product formula is important in applications as
it provides a way to approximate the (in general) complicated solution
$\e^{-t\oC}$ by subsequently applying the simpler parts $\e^{-t/n\oA}$
and $\e^{-t/n\oB}$ together $n$-times, and thus defining a numerical
split-step method.

In that paper, we are interested in the case where the underlying
Banach space $X$ is given by a product $X=X_{1}\times X_{2}$. The
operator of interest consists of two parts and is given in matrix
form
\[
-\cC=-\cA+\cB=\begin{pmatrix}-\oA_{1} & \oB_{1}\\
\oB_{2} & -\oA_{2}
\end{pmatrix},\quad\cA=\begin{pmatrix}\oA_{1} & \cdot\\
\cdot & \oA_{2}
\end{pmatrix},\quad\cB=\begin{pmatrix}\cdot & \oB_{1}\\
\oB_{2} & \cdot
\end{pmatrix},
\]
where $-\oA_{j}:\dom(\oA_{j})\subset X_{j}\to X_{j}$ are generators
of a semigroup, and $\oB_{1}:X_{2}\to\oX_{1}$, $\oB_{2}:X_{1}\to X_{2}$
are linear operators that describe the coupling. These matrix operators
occur frequently in physical problems: e.g. in continuum mechanics,
where $X_{1}$ contains the states in a bulk and $X_{2}$ contains
the states in a reservoir, the operators $\oA_{j}$ describe their
dynamics and $\oB_{j}$ the coupling; or in quantum mechanics, where
$X$ is the state space of two combined quantum systems, each individual
time-dependent system is described by $\oA_{j}$, the interaction
by $\oB_{1}$ and $\oB_{2}$.

We assume that both coupling operators $\oB_{j}$ are bounded (although
we discuss the unbounded case in Section \ref{sec:UnboundedCoupling}).
The operator $-\cA:\dom(\cA)=\dom(\oA_{1})\times\dom(\oA_{2})\to X$
is a generator of a semigroup on $X$, and its semigroup is given
by the diagonal matrix of the individual semigroups $\e^{-t\oA_{j}}$.
Moreover, $\cB$ is a bounded perturbation, and hence, $-\cC:\dom(\cC)=\dom(\cA)\to X$
is a generator, too. Its semigroup $\e^{-t\cC}$ defines the solution
$u=u(t)=(u_{1}(t),u_{2}(t))^{\mathrm{T}}$ of the abstract Cauchy
problem $\dot{u}(t)=-\cC u(t),\,u(0)=(x,y)^{\mathrm{T}}$ on $X$.
As an immediate consequence one can show that (under reasonable assumptions)
the solution operator of the combined system $\e^{-t\cC}$ can be
approximated by the Trotter-product formula $\left(\e^{-\tfrac{t}{n}\cA}\e^{\tfrac{t}{n}\cB}\right)^{n}$.
However, from the practical standpoint the Trotter-product formula
is often not useful because the semigroup of $\e^{t\cB}$ cannot be
expressed explicitly by the individual operators $\oB_{j}$. Moreover,
in each step an expensive evolution on the whole space $X=X_{1}\times X_{2}$
has to be calculated.

This problem can be solved by replacing $\e^{-\tfrac{t}{n}\cA}\e^{\tfrac{t}{n}\cB}$
by an split-step approximation that respects the underlying product
structure. The idea is to define a family of bounded operators $\cT=\cT(t)$
on $X$, consisting of two parts $\cT(t)=\cT_{2}(t)\cT_{1}(t)$ ,
where each bounded operator trajectory $\cT_{j}=\cT_{j}(t)$ defines
the solution of the abstract Cauchy problem $\dot{u}(t)=-\cC u(t),\,u(0)=(x,y)^{\mathrm{T}}$,
with one freezed (constant in time) component. These subsequent abstract
Cauchy problems become inhomogeneous and can be solved explicitly
(see Section \ref{subsec:Inhomogeneous-abstract-Cauchy}). The split-step
approximation operator is then given by
\begin{equation}
\cT(t)=\begin{pmatrix}\e^{-t\oA_{1}} & \int_{0}^{t}\e^{-s\oA_{1}}\oB_{1}\d s\\
\int_{0}^{t}\e^{-s\oA_{2}}\oB_{2}\d s\,\e^{-t\oA_{1}} & \int_{0}^{t}\e^{-s\oA_{1}}\oB_{1}\d s\,\int_{0}^{t}\e^{-s\oA_{2}}\oB_{2}\d s+\e^{-t\oA_{2}}
\end{pmatrix},\tag{AO}\label{eq:Approximation}
\end{equation}
which is the main object of investigation in the paper, see Section
\ref{sec:Split-step-method} for more details. Operator-matrix semigroups
have attracted a lot of attention in the last decades: in spectral
analysis \cite{Arli02SBOM,Tret08STBOM}; in modeling and solving various
types of evolution equations, see e.g. \cite{EngeOMSEE,BatPia05SDE,LiHuCh20SGUBOM,AgrHus21MRHCBOM};
and in the context of split-step methods \cite{CsoNic08OSDE,BCEF12SCPFOM,BCEF14SLTPOMS}.
For a recent split-step convergence analysis in the context of nonlinear
gradient-flow PDEs we refer to \cite{MiRoSt23TSMGF}.

Here, under the assumption that the coupling operators $\oB_{j}$
are bounded, a straight-forward computation shows that $\cT$ is stable,
i.e. for all $n\geq0$ and $t\geq0$, $\cT(t/n)^{n}$ can be bounded.
Stability provides that $\cT(t/n)^{n}$ is a well-defined approximation,
and, in particular, we have the strong-convergence result $\cT(t/n)^{n}x\to\e^{-t\cC}x$
for all $x\in X$ as $n\to\infty$, see Proposition \ref{prop:StabilityStrongConvergence}.
The main result (Theorem \ref{thm:MainConvergenceResultHolomorphicSemigroups})
is that, assuming that the semigroups of $-\oA_{1}$ and $-\oA_{2}$
are holomorphic, the convergence of the approximation can be improved
to operator-norm convergence, and moreover, can be estimated by $O(\log(n)/n)$,
which is almost the optimal convergence rate of $O(1/n)$. Operator-norm
convergence for approximations of Trotter-product form have been derived
first in \cite{Roga93EBTTFSAO} for semigroups of self-adjoint operators
in Hilbert spaces, and later generalized by \cite{CacZag01ONCTPFHS}
to holomorphic semigroups on Banach spaces. Here, we also assume that
the semigroups are holomorphic. However, the result here shows operator-norm
convergence for the approximation operators $\cT$ where its components
are not given by semigroups (although $\oB_{j}$ are bounded they
are not generators as they even act between different spaces). The
crucial idea is to estimate $\cT(t)-\e^{-t\cC}$ for small $t\geq0$
by evaluating their derivatives, and to show that this $O(t^{2})$
behavior propagates to the whole approximation. Technical difficulties
arise by the additional non-commutative feature of matrix multiplication.

The practical component of operator-norm convergence, in contrast
to convergence in the strong topology, is evident: The solution can
be approximated regardless the initial condition (and its generic
uncertainty); moreover, a convergence-rate estimate provides a universal
bound how good the approximation actually is. Moreover, the convergence
proof is flexible to provide also convergence for different (but similar)
approximation (see Section \ref{subsec:Other-approximations}).

The crucial assumption for estimating the convergence rate is the
holomorphicity of the dominating semigroups. In \cite{NeStZa18RONCTPF},
a counterexample has been constructed showing that operator-norm convergence
does not hold if the semigroup of the main operator is not holomorphic.
There, the construction is done by a time-dependent perturbation.
Affirmative convergence result for time-dependent perturbations can
be found in \cite{NeStZa17CREASO,NeStZa18ONCTPF,NeStZa19TPFLEEHS,NeStZa20CRETPA}.
In principle, time-dependent couplings $t\mapsto\oB_{j}(t)$ can also
be considered for approximations \eqref{eq:Approximation}. The corresponding
operator-norm convergence result however is left for future work.

\section{Preliminaries}

In this section, $\left(X,\|\cdot\|\right)$ is a general Banach space;
moreover, all operators in the paper are linear. We first recall well-known
important facts from semigroup theory and (inhomogeneous) abstract
Cauchy problems, see e.g. \cite{EngNag00OPSLEE}

\subsection{Recap of semigroup theory\label{subsec:Semigroup-theory}}

A family $\left\{ \oT(t)\right\} _{t\geq0}$ of bounded linear operators
on the Banach space $X$ is called a \textit{strongly continuous semigroup}
(in the following only \textit{semigroup}) if it satisfies the functional
equation 
\[
\cT(0)=\oI,\quad\oT(t+s)=\oT(t)\oT(s),\quad\text{for }t,s\geq0,
\]
and, moreover, the orbit maps $[0,\infty[\ni t\mapsto\oT(t)x$ are
continuous for all $x\in X$. In the following, the identity map is
denoted by $\oI:X\to X$. For a given semigroup its generator is a
linear operator defined by the limit 
\[
-\oA x:=\lim_{t\to0}\frac{1}{t}\left(\oT(t)x-x\right)
\]
on the domain $\dom(\oA)=\left\{ x\in X:\lim_{t\to0}\frac{1}{t}\left(\oT(t)x-x\right)\text{ exists}\right\} $.

It is well-known that the generator $-\oA$ of a strongly continuous
semigroup is a closed and densely defined linear operator, which uniquely
determines its semigroup, which we will denote by $\left\{ \oT(t)=\e^{-t\oA}\right\} _{t\geq0}$.
Recall that for a semigroup $\left\{ \oT(t)=\e^{-t\oA}\right\} _{t\geq0}$,
there are constants $M,\beta$ such that $\|\oT(t)\|\leq M\e^{\beta t}$
for all $t\geq0$. The operator norm for a bounded operator $\oB:X\to Y$
is defined as usual by $\|\oB\|:=\sup\left\{ \|\oB x\|_{Y}:\|x\|_{X}\leq1\right\} $,
which is a norm in the space of bounded linear operators. If $\beta\leq0$
then the semigroup is called \textit{bounded}; if $\|\oT(t)\|\leq1$,
the semigroup is called a \textit{contraction semigroup.}

The semigroup $\left\{ \oT(t)=\e^{-t\oA}\right\} _{t\geq0}$ is called
a \textit{bounded holomorphic semigroup} if its generator $-\oA$
satisfies $\oT(t)x\in\dom(\oA)$ for all $x\in X$ and $t>0$, and
if there is a constant $M_{A}>0$ such that $\sup_{t>0}\|t\oA\oT(t)\|\leq M_{A}$.
Recall that in this case the bounded semigroup $\left\{ \oT(t)\right\} _{t\geq0}$
has a unique analytic continuation into the open sector $\left\{ z\in\mathbb{C}\setminus\left\{ 0\right\} :|\mathrm{arg}(z)|<\delta(\oA)\leq\pi/2\right\} \subset\mathbb{C}$
of an angle $\delta(\oA)>0$.

In the following, we are interested in sums of operators given by
a generator $-\oA$ with semigroup $\|\e^{-t\oA}\|\leq M\e^{\beta t}$
and a bounded operator $\oB$. It is well-known that $-\oC:=-\oA+\oB$,
$\dom(\oC)=\dom(\oA)$ is a generator of a semigroup $\left\{ \e^{-t\oC}\right\} _{t\geq0}$
that satisfies $\|\e^{-t\oC}\|\leq M\e^{(\beta+M\|\oB\|)t}$ for all
$t\geq0$ (see \cite[Theorem III.1.3]{EngNag00OPSLEE}). Moreover,
if the semigroup $\left\{ \e^{-t\oA}\right\} _{t\geq0}$ is a holomorphic
semigroup, then so is $\left\{ \e^{-t\oC}\right\} _{t\geq0}$ (see
Proposition III.1.12).

We recall the following facts, which will be important for further
estimates.
\begin{lem}[Engel-Nagel, Lemma II.1.3]
\label{lem:PropertiesSemigroups}Let $-\oA$ be a generator of a
bounded semigroup $\left\{ \e^{-t\oA}\right\} _{t\geq0}$. Define
for $x\in X$ and $t\geq0$ the operator 
\[
\oF_{t}x:=\int_{0}^{t}\e^{-s\oA}x\d s.
\]
Then, we have
\begin{enumerate}
\item $\oF_{t}:X\to X$ is bounded with $\|\oF_{t}\|\leq tM$, and for all
$x\in X$ we have $\frac{1}{t}\oF_{t}x\to x$ as $t\to0$.
\item For all $t\geq0$ and $x\in X$, we have $\oF_{t}x\in\dom(\oA)$,
and $\e^{-t\oA}x-x=-\oA\oF_{t}x$. In particular, we have $\|\oA\oF_{t}\|\leq M+1$.
\item If, in addition, $\oA:\dom(\oA)\subset X\to X$ is boundedly invertible,
then $\oF_{t}x=\oA^{-1}\left(\oI-\e^{-t\oA}\right)x.$
\end{enumerate}
\end{lem}

\subsection{Inhomogeneous abstract Cauchy problems\label{subsec:Inhomogeneous-abstract-Cauchy}}

The following split-step method is based by solving inhomogeneous
abstract Cauchy problems, which are in general of the form
\begin{equation}
\dot{u}(t)=-\oA u(t)+f(t)\,,\quad u(0)=u_{0}\in X,\tag{iaCP}\label{eq:iACP}
\end{equation}
where $-\oA:\dom(A)\to X$ is a generator of a strongly continuous
semigroup, and $f:[0,\infty[\to X$ is an inhomogeneity. The formal
solution is then given by the variation of constants formula and has
the form
\[
u(t)=\e^{-t\oA}u_{0}+\int_{0}^{t}\e^{-(t-s)\oA}f(s)\d s.
\]
At least formally, one easily sees that 
\[
\dot{u}(t)=-\oA\e^{-t\oA}u_{0}-\oA\int_{0}^{t}\e^{-(t-s)\oA}f(s)\d s+f(t)=-\oA u(t)+f(t)\,.
\]
There are well-known criteria on how temporal as well as spatial regularity
of $f$ determine regularity of the solution $u$ of the inhomogeneous
Abstract Cauchy problem \eqref{eq:iACP}, see e.g. \cite[Chapter VI.7]{EngNag00OPSLEE}.

However, in our situation the formula simplifies as the inhomogeneity
$f$ will be constant in time. Indeed for $f(t)=x$, we get, with
a reparametrization of the integral
\begin{equation}
u(t)=\e^{-t\oA}u_{0}+\int_{0}^{t}\e^{-s\oA}x\d s=\e^{-t\oA}u_{0}+\oF_{t}x.\tag{\ensuremath{\star}}\label{eq:SolutionIACP}
\end{equation}
If moreover, the generator $\oA$ is boundedly invertible then the
integral can be solved using Lemma \ref{lem:PropertiesSemigroups},
and we have
\[
u(t)=\e^{-t\oA}u_{0}-\oA^{-1}\left(\e^{-t\oA}-\oI\right)x
\]
The following lemma collects and summarizes these important facts,
and is an immediate consequence of Lemma \ref{lem:PropertiesSemigroups}
and the classical well-posedness theory for linear (inhomogeneous)
abstract Cauchy problems (see e.g. Engel-Nagel Corollary VI.7.8)
\begin{lem}
\label{lem:ExistenceIACP}Let $x\in X$ and $u_{0}\in\dom(\oA)$ be
arbitrary. Define, for $t\in[0,\infty[$ the trajectory $u$ by \eqref{eq:SolutionIACP}.
Then $u$ is continuously differentiable, $u(t)\in X$ for all $t\geq0$
and is the unique classical solution of the inhomogeneous abstract
Cauchy problem $\dot{u}(t)=-\oA u(t)+x$.
\end{lem}

\section{Split-step method on the product space\label{sec:Split-step-method}}

After presenting the product-space setting, we define the split-step
approximation operators. Then, we show stability (i.e. boundedness)
of the time-discretized trajectories and, finally, discuss the first
convergence results.

\subsection{Product-space setting}

For two given Banach spaces $\left(X_{j},\|\cdot\|_{X_{j}}\right)$,
we consider the Banach space $X=X_{1}\times X_{2}$ equipped with
the canonical norm, i.e.
\[
\|(x,y)\|_{X_{1}\times X_{2}}^{2}:=\|x\|_{X_{1}}^{2}+\|y\|_{X_{2}}^{2}\,.
\]
The identity operator on the spaces $X_{j}$ is denoted by $\oI:X_{j}\to X_{j}$,
the identity operator on the whole space $X$ is denoted by $\cI:=\begin{pmatrix}\oI & \cdot\\
\cdot & \oI
\end{pmatrix}:X\to X$. For a bounded operator $\cB=\begin{pmatrix}\oB_{1} & \oB_{12}\\
\oB_{21} & \oB_{2}
\end{pmatrix}:X\to X$, where $\oB_{j}:X_{j}\to X_{j}$ and $\oB_{ij}:X_{j}\to X_{i}$ are
bounded operators, an easy calculation shows
\begin{align*}
\|\cB\begin{pmatrix}x\\
y
\end{pmatrix}\|^{2} & =\|\oB_{1}x+\oB_{12}y\|_{X_{1}}^{2}+\|\oB_{21}x+\oB_{2}y\|_{X_{2}}^{2}\leq\\
 & \leq2\left(\|\oB_{1}\|^{2}+\|\oB_{21}\|^{2}\right)\|x\|_{X_{1}}^{2}+\left(\|\oB_{12}\|^{2}+\|\oB_{2}\|^{2}\right)\|y\|_{X_{1}}^{2}\\
 & \leq2\max\left\{ \|\oB_{1}\|^{2}+\|\oB_{21}\|^{2},\|\oB_{12}\|^{2}+\|\oB_{2}\|^{2}\right\} \left(\|x\|_{X_{1}}^{2}+\|y\|_{X_{1}}^{2}\right),
\end{align*}
which implies the crude estimate $\|\cB\|\leq\sqrt{2}\max\left\{ \left(\|\oB_{1}\|^{2}+\|\oB_{21}\|^{2}\right)^{1/2},\left(\|\oB_{12}\|^{2}+\|\oB_{2}\|^{2}\right)^{1/2}\right\} $.

\subsection{Split-step by inhomogeneous abstract Cauchy problem}

Now, we describe in detail the split-step method for approximating
the solution on a product space. We consider on the space $X=X_{1}\times X_{2}$
the operator
\[
-\cC:=-\cA+\cB,\quad\cA=\begin{pmatrix}\oA_{1} & \cdot\\
\cdot & \oA_{2}
\end{pmatrix},\quad\cB=\begin{pmatrix}\cdot & \oB_{1}\\
\oB_{2} & \cdot
\end{pmatrix},\quad-\cC\begin{pmatrix}u\\
v
\end{pmatrix}=\begin{pmatrix}-\oA_{1}u+\oB_{1}v\\
\oB_{2}u-\oA_{2}v
\end{pmatrix},
\]
where each $-\oA_{j}:\dom(\oA_{j})\subset X_{j}\to X_{j}$ are generators
of a contraction semigroup and the coupling operators $\oB_{1}:X_{2}\to X_{1}$
and $\oB_{2}:X_{1}\to X_{2}$ are bounded. (We comment on unbounded
operators in Section \ref{sec:UnboundedCoupling}.)

Clearly, also $-\cA:\dom(\cA)=\dom(\oA_{1})\times\dom(\oA_{2})\subset X_{1}\times X_{2}\to X$
is a generator of a semigroup $\e^{-t\cA}=\begin{pmatrix}\e^{-t\oA_{1}} & \cdot\\
\cdot & \e^{-t\oA_{2}}
\end{pmatrix}$, which is also a contraction semigroup. Since $\cB$ is bounded,
we have that $-\cC:\dom\left(\cC\right)=\dom\left(\cA\right)\subset X\to X$
is also a generator of a semigroup.

Here, we will not approximate the solution operator $\e^{-t\cC}$
by the Trotter-product formula $\left(\e^{-t/n\cA}\e^{t/n\cB}\right)^{n}$.
Instead, we are interested in an approximation that exploits the Block
structure of the underlying state space. As for the Trotter product
formula, the time-discretized iteration operator $\cT(\tau)$ depends
on the small time-step of length $\tau=t/n$ and consists of two parts,
i.e $\cT=\cT_{2}\cT_{1}$. Each operator $\cT_{j}$ is defined by
evolving only the component of $X_{j}$ and leaving the other component
constant. The first operator $\cT_{1}$ is given by solving the inhomogeneous
abstract Cauchy problem:
\begin{align*}
\begin{cases}
\dot{u}= & -\oA_{1}u+\oB_{1}v\\
\dot{v}= & 0
\end{cases} & ,\text{ for }t\in[0,\infty[,\quad u(0)=u_{0},\,v(0)=v_{0},
\end{align*}
 and maps $\cT_{1}(\tau):\begin{pmatrix}u_{0}\\
v_{0}
\end{pmatrix}\mapsto\begin{pmatrix}u(\tau)\\
v(\tau)
\end{pmatrix}$. The explicit solution of the evolution equation is given by Lemma
\ref{lem:ExistenceIACP}, and we have
\[
\begin{pmatrix}u\\
v
\end{pmatrix}(t=\tau)=\cT_{1}(\tau)\begin{pmatrix}u_{0}\\
v_{0}
\end{pmatrix}=\begin{pmatrix}\e^{-\tau\oA_{1}} & \int_{0}^{\tau}\e^{-\sigma\oA_{1}}\d\sigma\oB_{1}\\
\cdot & \oI
\end{pmatrix}\begin{pmatrix}u_{0}\\
v_{0}
\end{pmatrix}\,.
\]
The second operator $\cT_{2}$ is given by solving the inhomogeneous
abstract Cauchy problem:
\begin{align*}
\begin{cases}
\dot{u}= & 0\\
\dot{v}= & \oB_{2}u-\oA_{2}v
\end{cases} & ,\text{ for }t\in[0,\infty[,\quad u(0)=u_{0},\,v(0)=v_{0},
\end{align*}
and maps $\cT_{2}(\tau):\begin{pmatrix}u_{0}\\
v_{0}
\end{pmatrix}\mapsto\begin{pmatrix}u(\tau)\\
v(\tau)
\end{pmatrix}$. Hence, we have 
\[
\begin{pmatrix}u\\
v
\end{pmatrix}(t=\tau)=\cT_{2}(\tau)\begin{pmatrix}u_{0}\\
v_{0}
\end{pmatrix}=\begin{pmatrix}\oI & \cdot\\
\int_{0}^{\tau}\e^{-\sigma\oA_{2}}\d\sigma\oB_{2} & \e^{-\tau\oA_{2}}
\end{pmatrix}\begin{pmatrix}u_{0}\\
v_{0}
\end{pmatrix}\,.
\]
Hence, we conclude that the total solution operator at time $t=\tau$
is given by 
\begin{align}
\cT(\tau) & =\cT_{2}(\tau)\cT_{1}(\tau)=\begin{pmatrix}\oI & \cdot\\
\int_{0}^{\tau}\e^{-\sigma\oA_{2}}\d\sigma\oB_{2} & \e^{-\tau\oA_{2}}
\end{pmatrix}\begin{pmatrix}\e^{-\tau\oA_{1}} & \int_{0}^{\tau}\e^{-\sigma\oA_{1}}\d\sigma\oB_{1}\\
\cdot & \oI
\end{pmatrix}\nonumber \\
 & =\begin{pmatrix}\e^{-\tau\oA_{1}} & \int_{0}^{\tau}\e^{-\sigma\oA_{1}}\d\sigma\oB_{1}\\
\int_{0}^{\tau}\e^{-\sigma\oA_{2}\d\sigma}\oB_{2}\e^{-\tau\oA_{1}} & \int_{0}^{\tau}\e^{-\sigma\oA_{2}}\d\sigma\oB_{2}\cdot\int_{0}^{\tau}\e^{-\sigma\oA_{1}}\d\sigma\oB_{1}+\e^{-\tau\oA_{2}}
\end{pmatrix}\nonumber \\
 & =:\begin{pmatrix}\oE_{1}(\tau) & \oX_{1}(\tau)\\
\oX_{2}(\tau)\oE_{1}(\tau) & \oX_{2}(\tau)\oX_{1}(\tau)+\oE_{2}(\tau)
\end{pmatrix},\tag{AO}\label{eq:Definition=00005CcT}
\end{align}
where we have introduced the notation for the solution operator $\oE_{j}$
on the diagonal and the cross terms $\oX_{j}$
\[
\oE_{j}(\tau)=\e^{-\tau\oA_{j}},\quad\oX_{j}(\tau)=\int_{0}^{\tau}\e^{-\sigma\oA_{j}}\d\sigma\oB_{j}.
\]
Some easy facts regarding these operators are summarized in the next
lemma, which is an trivial consequence of Lemma \ref{lem:PropertiesSemigroups}.
\begin{lem}
\label{lem:Properties=00005CcT}Let $-\oA_{j}:\dom(\oA_{j})\subset X_{j}\to X_{j}$
be generators of contraction semigroups, and let $\oB_{1}:X_{2}\to X_{1}$,
$\oB_{2}:X_{1}\to X_{2}$ be bounded. Define the split-step approximation
operators $\left\{ \cT(\tau)\right\} _{\tau\geq0}$ by \eqref{eq:Definition=00005CcT}.
Then
\begin{enumerate}
\item For all $\tau\geq0$, the operators $\cT(\tau):X\to X$ are strongly
continuous bounded operators and $\cT(0)=\cI$,
\item For all $\tau\geq0$, we have $\|\oX_{i}(\tau)\|\leq\tau\|\oB_{i}\|$,
and there is a constant such that for all $\tau\geq0$, we have $\|\oA_{i}\oX_{i}(\tau)\|\leq2\|\oB_{i}\|$.
\item For all $x_{j}\in\dom(\oA_{j})$, we have $\oE_{j}'(\tau)x_{j}=-\oA_{j}\e^{-\tau\oA_{j}}x_{j}$;
For all $x\in X_{j}$ we have $\oX_{j}'(\tau)x=\oE_{j}(\tau)\oB_{j}$.
\end{enumerate}
\end{lem}

Iterating the operators $\cT(\tau)$ $n$-times, we get a trajectory
till time $t\in[0,\infty[$. The main question, which is addressed
in that paper, is to show convergence
\[
\cT(t/n)^{n}\to\e^{-t\cC}.
\]

\subsection{Stability analysis and convergence in the strong topology}

To have a useful approximation, the stability of the iterated operator
$\cT(t/n)^{n}$ for $t\in[0,\infty[$ has to be established.
\begin{prop}
\label{prop:StabilityStrongConvergence}Let $-\oA_{j}:\dom(\oA_{j})\subset X_{j}\to X_{j}$
be generators of strongly continuous contraction semigroups, and let
$\oB_{1}:X_{2}\to X_{1}$, $\oB_{2}:X_{1}\to X_{2}$ be bounded. Then
for all $t\in[0,\infty[$ and $n\in\N$ we have
\[
\|\cT(t/n)^{n}\|\leq\e^{t\left(\|\oB_{1}\|+\|\oB_{2}\|\right)}.
\]
Moreover, for all $t\in[0,\infty[$, $\left(\cT(t/n)^{n}\right)_{n\in\N}$
converges to the semigroup $\e^{-t\cC}$ in the strong topology, i.e.
for all $x\in X$, we have
\[
\lim_{n\to\infty}\cT(t/n)^{n}x=\e^{-t\cC}x,
\]
which is uniformly in time on bounded intervals.
\end{prop}

\begin{proof}
Using that $-\oA_{j}$ generates a contraction semigroup, we have
\begin{align*}
\|\cT_{1}(\tau)\begin{pmatrix}x\\
y
\end{pmatrix}\|^{2} & =\|\oE_{1}(\tau)x+\oX_{1}(\tau)y\|_{X_{1}}^{2}+\|y\|_{X_{2}}^{2}\leq\left(\|\oE_{1}(\tau)x\|_{X_{1}}+\|\oX_{1}(\tau)y\|_{X_{1}}\right)^{2}+\|y\|_{X_{2}}^{2}\\
 & \leq\|x\|_{X_{1}}^{2}+2\tau\|\oB_{1}\|\cdot\|x\|_{X_{1}}\cdot\|y\|_{X_{2}}+\left(1+\tau^{2}\|\oB_{1}\|^{2}\right)\|y\|_{X_{2}}^{2}\\
 & \leq\left(1+2\tau\|\oB_{1}\|+\tau^{2}\|\oB_{1}\|^{2}\right)\left(\|x\|_{X_{1}}^{2}+\|y\|_{X_{2}}^{2}\right)=\left(1+\tau\|\oB_{1}\|\right)^{2}\left(\|x\|_{X_{1}}^{2}+\|y\|_{X_{2}}^{2}\right).
\end{align*}
Hence, we get that $\|\cT_{1}(\tau)\|\leq1+\tau\|\oB_{1}\|$ for all
$\tau\geq0$. Similarly, we obtain $\|\cT_{2}(\tau)\|\leq1+\tau\|\oB_{2}\|$
for all $\tau\geq0$. Hence, we get
\begin{align*}
\|\cT(t/n)^{n}\|\leq\|\cT_{2}(t/n)\|^{n}\cdot\|\cT_{1}(t/n)\|^{n} & \leq\left(1+\frac{t}{n}\|\oB_{2}\|\right)^{n}\left(1+\frac{t}{n}\|\oB_{1}\|\right)^{n}\leq\e^{t\left(\|\oB_{1}\|+\|\oB_{2}\|\right)}\,.
\end{align*}
To prove convergence in the strong topology, we rely on the famous
result of Chernoff, see e.g. \cite[Corollary III.5.3]{EngNag00OPSLEE}.
Since the stability has been already shown, it suffices to show that
the derivative of $\cT(t)$ at $t\geq0$ is given by $-\cC$ (which,
as we already know, is a generator). For this let $\begin{pmatrix}u,v\end{pmatrix}^{\mathrm{T}}\in\dom\left(\cC\right)=\dom\left(\cA\right)$.
By Lemma \ref{lem:Properties=00005CcT}, we obtain
\begin{align*}
\frac{1}{t}\left(\cT(t)-\cI\right)\begin{pmatrix}u\\
v
\end{pmatrix} & =\frac{1}{t}\begin{pmatrix}\oE_{1}(t)-\oI & \oX_{1}(t)\\
\oX_{2}(t)\oE_{1}(t) & \oX_{2}(t)\oX_{1}(t)+\oE_{2}(t)-\oI
\end{pmatrix}\begin{pmatrix}u\\
v
\end{pmatrix}\\
 & \to\begin{pmatrix}-\oA_{1}u+\oB_{1}v\\
\oB_{2}u-\oA_{1}v
\end{pmatrix}=-\cC\begin{pmatrix}u\\
v
\end{pmatrix},
\end{align*}
as $t\to0$. This implies $\lim_{n\to\infty}\cT(t/n)^{n}x=\e^{-t\cC}x$
as desired.
\end{proof}
The main result of the paper is to show that the convergence of the
approximation $\cT(t/n)^{n}\to\e^{-t\cC}$ can be improved to convergence
in the operator norm, and that the convergence can be estimated. To
show the ideas, we shortly discuss the situation, when $-\oA_{j}$
are bounded operators.

\subsection{Convergence for bounded operators}

Assuming that $-\oA_{j}$ are bounded operators, the semigroups are
given by the exponential of the generators, and we have
\[
\e^{-\tau\mathbb{\oA}_{j}}=\oI-\tau\oA_{j}+O(\tau^{2}),\quad\text{as}\quad\tau\to0\,.
\]
As we will see, estimating $\cT(t/n)^{n}\to\e^{-t\cC}$ is similar
to the unbounded case and relies on the telescopic representation
of the product:
\[
\mathbb{\cT}^{n}-\mathbb{\cS}^{n}=\sum_{k=0}^{n-1}\mathbb{\cT}^{n-1-k}(\cT-\mathbb{\cS})\mathbb{\cS}^{k}.
\]
The crucial idea is to get a good convergence rate for $\cT-\mathbb{\cS}$,
and to ensure that the remainders $\mathbb{\cT}^{n-1-k}$ and $\mathbb{\cS}^{k}$
can be bounded.

Indeed, in the situation of bounded operators we have, for $\tau\to0$,
\begin{align*}
\cT(\tau)-\e^{-\tau\cC} & =\begin{pmatrix}\oE_{1}(\tau) & \oX_{1}(\tau)\\
\oX_{2}(\tau)\oE_{1}(\tau) & \oX_{2}(\tau)\oX_{1}(\tau)+\oE_{2}(\tau)
\end{pmatrix}-\left\{ \cI-\tau\cC\right\} +O(\tau^{2})\\
 & =\begin{pmatrix}\oE_{1}(\tau) & \oX_{1}(\tau)\\
\oX_{2}(\tau)\oE_{1}(\tau) & \oX_{2}(\tau)\oX_{1}(\tau)+\oE_{2}(\tau)
\end{pmatrix}-\begin{pmatrix}\oI-\tau\oA_{1} & \tau\oB_{1}\\
\tau\oB_{2} & \oI-\tau\oA_{2}
\end{pmatrix}+O(\tau^{2})=O(\tau^{2}),
\end{align*}
where we have used that 
\begin{align*}
\oE_{j}(\tau) & =\oI-\tau\oA_{j}+O(\tau^{2}),\quad\oX_{j}(\tau)=\tau\oB_{j}+O(\tau^{2}).
\end{align*}
Hence, we get 
\begin{align*}
\|\cT(\tau)^{n}-\e^{-t\cC}\| & =\|\cT(\tau)^{n}-\left(\e^{-\tau\cC}\right)^{n}\|\\
 & =\|\sum_{k=0}^{n-1}\mathbb{\cT}(\tau)^{n-1-k}(\cT(\tau)-\e^{-\tau\cC})\e^{-\tau k\cC}\|\\
 & \leq\sum_{k=0}^{n-1}\|\mathbb{\cT}(\tau)^{n-1-k}\|\cdot\|\cT(\tau)-\e^{-\tau\cC}\|\cdot\|\e^{-\tau k\cC}\|\lesssim n\cdot\tau^{2}\lesssim\frac{t^{2}}{n}\,.
\end{align*}
Hence, the split-step method converges with order $O(n^{-1})$, which
is summarized in the next proposition.
\begin{prop}
Let $\oA_{i}:X_{i}\to X_{i}$, $\oB_{1}:X_{2}\to X_{1}$ and $\oB_{2}:X_{1}\to X_{2}$
be bounded. Define the split-step approximation operator family $\cT$
by \eqref{eq:Definition=00005CcT}. Then there is a constant $C=C(\|A\|,\|B\|)$
such that for all $t\geq0$ and $n\in\N$ we have $\|\cT(\tau)^{n}-\e^{-t\cC}\|\leq\frac{C}{n}t^{2}$.
\end{prop}

We note that, in general no better convergence than of order $O(n^{-1})$
can be expected. Moreover, the above calculation suggest that $\cT_{\mathrm{B}}(\tau):=\begin{pmatrix}\oE_{1}(\tau) & \tau\oB_{1}\\
\tau\oB_{2} & \oE_{2}(\tau)
\end{pmatrix}$, where the integrand $\e^{-\sigma\oA_{j}}$ in the coupling term
$\oX_{j}(\tau)$ is replaced by the constant identity $\oI$, can
be used as another approximation. Indeed, following the lines of the
proof of Proposition \ref{prop:StabilityStrongConvergence} the approximation
based on $\cT_{\mathrm{B}}$ converges on the strong topology as well.
However, for proving convergence in operator norm the regularization
$\e^{-\sigma\oA_{j}}$ is needed (see also Section \ref{subsec:Other-approximations},
where other different approximations are discussed).

\section{Operator-norm convergence rate analysis}

In this section, we show the convergence $\cT(t/n)^{n}\to\e^{-t\cC}$
in the operator norm as $n\to\infty$, and that the convergence rate
can be estimated. The overall assumptions is that the operators $-\oA_{j}:\dom\left(\oA_{j}\right)\to X_{j}$
are generators of holomorphic contraction semigroups and that the
linear operators $\oB_{1}:X_{2}\to X_{1}$ and $\oB_{2}:X_{1}\to X_{2}$
are bounded. In particular, there is a constant $M_{A}>0$ such that
for all $t>0$
\begin{equation}
\|t\oA_{i}\e^{-t\oA_{i}}\|\leq M_{A}\,,\label{eq:EstimateHolomorphicSemigroups=00005CoA}
\end{equation}
which we will use frequently. Since $\cA$ is diagonal, we easily
see that $-\cA:\dom(\cA)\subset X\to X$ is a generator of a contraction
semigroup. In addition, we have the same estimate $\|t\cA\e^{-t\cA}\|\leq M_{A}$,
which shows that the semigroup $\e^{-t\cA}$ is holomorphic as well.
Since the coupling operator $\cB$ is bounded, also $-\cC$ is a generator
of a semigroup with $\|\e^{-t\cC}\|\leq\e^{t\|\cB\|}$ which is also
holomorphic, i.e. there is a constant $M_{C}>0$ such that for all
$t>0$ we have $\|t\cC\e^{-t\cC}\|\leq M_{C}$.

Moreover, we will use the following relations between the operator
norms of $\oB_{j}$ and $\cB$:
\[
\|\cB\|=\max\left\{ \|\oB_{1}\|,\|\oB_{2}\|\right\} \leq\|\oB_{1}\|+\|\oB_{2}\|\leq2\|\cB\|\,.
\]

\subsection{Auxiliary lemmas}

We first collect important auxiliary lemmas which are used to show
that main result. %
We frequently use the matrix multiplication rule for operators (wherever
they are defined):
\[
{\cal C}{\cal D}=\begin{pmatrix}\oC_{1} & \oC_{2}\\
\oC_{3} & \oC_{4}
\end{pmatrix}\begin{pmatrix}\mathbf{D}_{1} & \mathbf{D}_{2}\\
\mathbf{D}_{3} & \mathbf{D}_{4}
\end{pmatrix}=\begin{pmatrix}\oC_{1}\oD_{1}+\oC_{2}\oD_{3} & \oC_{1}\oD_{2}+\oC_{2}\oD_{4}\\
\oC_{3}\oD_{1}+\oC_{4}\oD_{3} & \oC_{3}\oD_{2}+\oC_{4}\oD_{4}
\end{pmatrix}\,.
\]

\begin{lem}
\label{lem:AuxLemmaHolomorphicSemigroups}Let $-\oA_{j}:\dom\left(\oA_{j}\right)\subset X_{j}\to X_{j}$
be generators of holomorphic contraction semigroups, and let $\oB_{1}:X_{2}\to X_{1}$,
$\oB_{2}:X_{1}\to X_{2}$ be bounded. Then, there is a constant $C_{1}=C_{1}(\|\oB_{1}\|,\|\oB_{2}\|)>0$
such that for all $k\in\left\{ 1,\dots,n\right\} $ and $\tau=\frac{t}{n}>0$
we have 
\[
\|\cA\cT(\tau)^{k}\|\leq C_{1}\e^{t\left(\|\oB_{1}\|+\|\oB_{2}\|\right)}\left(1+\log k\right)+\frac{M_{A}}{k\tau}.
\]
\end{lem}

\begin{proof}
Using the decomposition
\[
\cA\cT(\tau)^{k}=\cA\left(\cT(\tau)^{k}-\e^{-\tau k\cA}\right)+\cA\e^{-\tau k\cA},
\]
we see that the second term is bounded by $\|\cA\e^{-\tau k\cA}\|\leq\frac{1}{\tau k}M_{A}$.
For the first term, we use the telescopic representation of the product
$\mathbb{\cS}^{k}-\mathbb{\cT}^{k}=\sum_{j=0}^{k-1}\mathbb{\cS}^{k-1-j}(\mathbb{\cS}-\cT)\mathbb{\cT}^{j},$
and get
\begin{align*}
\cA\left(\e^{-\tau k\cA}-\cT(\tau)^{k}\right) & =\sum_{j=0}^{k-1}\e^{-\tau(k-1-j)\cA}\cA\left(\e^{-\tau\cA}-\cT(\tau)\right)\cT(\tau)^{j}=\\
 & =\sum_{j=0}^{k-2}\e^{-\tau(k-1-j)\cA}\cA\left(\e^{-\tau\cA}-\cT(\tau)\right)\cT(\tau)^{j}+\cA\left(\e^{-\tau\cA}-\cT(\tau)\right)\cT(\tau)^{k-1}\,.
\end{align*}
Since
\begin{align*}
\cT(\tau)-\e^{-\tau\cA} & =\begin{pmatrix}\cdot & \oX_{1}(\tau)\\
\oX_{2}\oE_{1}(\tau) & \oX_{2}\oX_{1}(\tau)
\end{pmatrix},
\end{align*}
we get that there is a constant $C=C(\|\oB_{1}\|,\|\oB_{2}\|)>0$
such that $\|\cT(\tau)-\e^{-\tau\cA}\|\leq C\tau$. Moreover, we see
that 
\[
\cA\left(\e^{-\tau\cA}-\cT(\tau)\right)=\begin{pmatrix}\cdot & \oA_{1}\oX_{1}(\tau)\\
\oA_{2}\oX_{2}\oE_{1}(\tau) & \oA_{2}\oX_{2}\oX_{1}(\tau)
\end{pmatrix},
\]
which is (by Lemma \ref{lem:Properties=00005CcT}) a bounded operator
for $\tau>0$ with $\|\cA\left(\cT(\tau)-\e^{-\tau\cA}\right)\|\leq C\tau$,
$C=C(\|\oB_{1}\|,\|\oB_{2}\|)>0$. Hence, we have that
\begin{align*}
\|\cA & \left(\e^{-\tau k\cA}-\cT(\tau)^{k}\right)\|\\
 & \leq\sum_{j=0}^{k-2}\|\e^{-\tau(k-1-j)\cA}\cA\|\cdot\|\e^{-\tau\cA}-\cT(\tau)\|\cdot\|\cT(\tau)^{j}\|+\|\cA\left(\e^{-\tau\cA}-\cT(\tau)\right)\|\cdot\|\cT(\tau)^{k-1}\|\\
 & \leq C\e^{t\left(\|\oB_{1}\|+\|\oB_{2}\|\right)}\left(\sum_{j=0}^{k-2}\frac{1}{\left(k-1-j\right)\tau}\tau+1\right)\leq C\e^{t\left(\|\oB_{1}\|+\|\oB_{2}\|\right)}\left(\log k+1\right)\,,
\end{align*}
where we have used $\sum_{j=1}^{k-1}\tfrac{1}{j}\leq\log k$. This
proves the claim.
\end{proof}
For the next two lemmas, we do not assume that the semigroups $\e^{-t\oA_{j}}$
are holomorphic. We recall that if $\cA$ and $\cC$ are boundedly
invertible, then the operators $\cA^{-1}\cC$, $\cC\cA^{-1}$ and
their inverses are all bounded.
\begin{lem}
\label{lem:AuxLemmaO(tau)Estimate}Let$-\oA_{j}:\dom\left(\oA_{j}\right)\subset X_{j}\to X_{j}$
be generators of contraction semigroups, and let $\oB_{1}:X_{2}\to X_{1}$,
$\oB_{2}:X_{1}\to X_{2}$ be bounded. Moreover let $\cA$ and $\cC$
be boundedly invertible. Then, there is a constant $C_{2}=C_{2}(\|\cA^{-1}\|,\|\cA^{-1}\cB\|)>0$
such that for all $\tau>0$ we have
\[
\|\cA^{-1}\left(\cT(\tau)-\e^{-\tau\cC}\right)\|\leq C_{2}\left(1+\e^{\tau\|\cB\|}\right)\tau.
\]
\end{lem}

\begin{proof}
We have
\[
\cT(\tau)-\e^{-\tau\cC}=\cT_{2}(\tau)\cT_{1}(\tau)-\e^{-\tau\cC}=\left(\cT_{2}(\tau)-\cI\right)\cT_{1}(\tau)+\cT_{1}(\tau)-\cI+\cI-\e^{-\tau\cC},
\]
and
\begin{align*}
\cA^{-1}\left(\cT_{2}(\tau)-\cI\right)\cT_{1}(\tau) & =\begin{pmatrix}\cdot & \cdot\\
\oA_{2}^{-1}\oX_{2}(\tau)\oE_{1}(\tau) & \oA_{2}^{-1}\oX_{2}(\tau)\oX_{1}(\tau)+\oA_{2}^{-1}\left(\oE_{2}(\tau)-\oI\right)
\end{pmatrix}\\
\cA^{-1}\left(\cT_{1}(\tau)-\cI\right) & =\begin{pmatrix}\oA_{1}^{-1}\left(\oE_{1}(\tau)-\oI\right) & \oA_{1}^{-1}\oX_{1}(\tau)\\
\cdot & \cdot
\end{pmatrix}\\
\cA^{-1}\left(\cI-\e^{-\tau\cC}\right) & =\cA^{-1}\cC\cC^{-1}\left(\cI-\e^{-\tau\cC}\right).
\end{align*}
Since $\|\oA_{j}^{-1}\left(\oE_{j}(\tau)-\oI\right)\|\leq\tau$ (see
Lemma \ref{lem:PropertiesSemigroups}), we get that $\|\cA^{-1}\left(\cT_{2}(\tau)-\cI\right)\cT_{1}(\tau)\|\leq C\tau$
and $\|\cA^{-1}\left(\cT_{1}(\tau)-\cI\right)\|\leq C\tau$, where
$C=C(\|\cA^{-1}\|)$. Moreover, we have 
\[
\|\cA^{-1}\left(\cI-\e^{-\tau\cC}\right)\|\leq\|\cA^{-1}\left(\cA-\cB\right)\|\cdot\e^{\tau\|\cB\|}\tau\leq\left(1+\|\cA^{-1}\cB\|\right)\cdot\e^{\tau\|\cB\|}\tau.
\]
\end{proof}
\begin{lem}
\label{lem:AuxLemmaO(tau2)Estimate}Let $-\oA_{j}:\dom\left(\oA_{j}\right)\subset X_{j}\to X_{j}$
be generators of contraction semigroups, and let $\oB_{1}:X_{2}\to X_{1}$,
$\oB_{2}:X_{1}\to X_{2}$ be bounded. Moreover let $\cA$ and $\cC$
be boundedly invertible. Then, there is a constant $C_{3}=C_{3}(\|\oB_{2}\|,\|\oB_{1}\|,\|\oA_{1}^{-1}\|,\|\oA_{2}^{-1}\|)>0$
such that for all $\tau>0$ we have
\[
\|\left(\cT(\tau)-\e^{-\tau\cC}\right)\cA^{-1}\|\leq C_{3}\tau^{2}\e^{\tau\|\cB\|}\ .
\]
\end{lem}

\begin{proof}
For better readability we neglect the $\tau$-dependance for a moment.
We have the decomposition
\[
\cT_{2}\cT_{1}-\e^{-\tau\cC}=\left(\cI-\cT_{2}\right)\left(\cI-\cT_{1}\right)+\cT_{1}+\cT_{2}-\cI-\e^{-\tau\cC}.
\]
The first term has the form
\begin{align*}
\left(\cI-\cT_{2}\right)\left(\cI-\cT_{1}\right) & =\begin{pmatrix}\cdot & \cdot\\
\oX_{2} & \oI-\oE_{2}
\end{pmatrix}\begin{pmatrix}\oI-\oE_{1} & \oX_{1}\\
\cdot & \cdot
\end{pmatrix}=\begin{pmatrix}\cdot & \cdot\\
\oX_{2}\left(\oI-\oE_{1}\right) & \oX_{2}\oX_{1}
\end{pmatrix},
\end{align*}
which is already of order $O(\tau)$. Moreover, we have 
\[
\left(\cI-\cT_{2}\right)\left(\cI-\cT_{1}\right)\cA^{-1}=\begin{pmatrix}\cdot & \cdot\\
\oX_{2}\left(\oI-\oE_{1}\right)\oA_{1}^{-1} & \oX_{2}\oX_{1}\oA_{2}^{-1}
\end{pmatrix},
\]
which shows that there is a constant $C=C(\|\oB_{2}\|,\|\oB_{1}\|,\|\oA_{2}^{-1}\|)$
such that the estimate $\|\left(\cI-\cT_{2}(\tau)\right)\left(\cI-\cT_{1}(\tau)\right)\cA^{-1}\|\leq C\tau^{2}$
holds.

It suffices to estimate the remaining part. We have
\[
\cT_{1}(\tau)+\cT_{2}(\tau)-\cI=\begin{pmatrix}\oE_{1}(\tau) & \oX_{1}(\tau)\\
\oX_{2}(\tau) & \oE_{2}(\tau)
\end{pmatrix}=:\widehat{\cT}(\tau)\,,
\]
where we have introduced the symmetric version $\widehat{\cT}$ of
$\cT_{2}\cT_{1}$. To estimate the difference $\widehat{\cT}(\tau)-\e^{-\tau\cC}$,
we rely on the following form
\begin{align*}
\left(\widehat{\cT}(\tau)-\e^{-\tau\cC}\right) & =\int_{0}^{\tau}\frac{\d}{\d\sigma}\left\{ \widehat{\cT}(\sigma)\e^{-(\tau-\sigma)\cC}\right\} \d\sigma=\\
 & =\int_{0}^{\tau}\left\{ \widehat{\cT}'(\sigma)\e^{-(\tau-\sigma)\cC}+\widehat{\cT}(\sigma)\cC\e^{-(\tau-\sigma)\cC}\right\} \d\sigma=\int_{0}^{\tau}\left\{ \widehat{\cT}'(\sigma)+\widehat{\cT}(\sigma)\cC\right\} \e^{-(\tau-\sigma)\cC}\d\sigma,
\end{align*}
which holds on $\dom(\cA)$.

We compute (see Lemma \ref{lem:Properties=00005CcT})
\[
\frac{\d}{\d\sigma}\widehat{\cT}(\sigma)=\frac{\d}{\d\sigma}\begin{pmatrix}\oE_{1}(\sigma) & \oX_{1}(\sigma)\\
\oX_{2}(\sigma) & \oE_{2}(\sigma)
\end{pmatrix}=\begin{pmatrix}-\oA_{1}\oE_{1}(\sigma) & \oE_{1}(\sigma)\oB_{1}\\
\oE_{2}(\sigma)\oB_{2} & -\oA_{2}\oE_{2}(\sigma)
\end{pmatrix},
\]
which provides the explicit simple form
\begin{align*}
\widehat{\cT}'(\sigma)+\widehat{\cT}(\sigma)\cC & =\begin{pmatrix}-\oA_{1}\oE_{1}(\sigma) & \oE_{1}(\sigma)\oB_{1}\\
\oE_{2}(\sigma)\oB_{2} & -\oA_{2}\oE_{2}(\sigma)
\end{pmatrix}+\begin{pmatrix}\oE_{1}(\sigma) & \oX_{1}(\sigma)\\
\oX_{2}(\sigma) & \oE_{2}(\sigma)
\end{pmatrix}\begin{pmatrix}\oA_{1} & -\oB_{1}\\
-\oB_{2} & \oA_{2}
\end{pmatrix}\\
 & =\begin{pmatrix}-\oX_{1}(\sigma)\oB_{2} & \oX_{1}(\sigma)\oA_{2}\\
\oX_{2}(\sigma)\oA_{1} & -\oX_{2}(\sigma)\oB_{1}
\end{pmatrix}=\begin{pmatrix}\oX_{1}(\sigma) & \cdot\\
\cdot & \oX_{2}(\sigma)
\end{pmatrix}\begin{pmatrix}-\oB_{2} & \oA_{2}\\
\oA_{1} & -\oB_{1}
\end{pmatrix}.
\end{align*}
In particular, we have
\[
\left(\widehat{\cT}'(\sigma)+\widehat{\cT}(\sigma)\cC\right)\cA^{-1}=\begin{pmatrix}\oX_{1}(\sigma) & \cdot\\
\cdot & \oX_{2}(\sigma)
\end{pmatrix}\begin{pmatrix}-\oB_{2}\cA_{1}^{-1} & \oI\\
\oI & -\oB_{1}\cA_{2}^{-1}
\end{pmatrix}\,.
\]
Hence, we get that there is a constant $C=C(\|\oB_{1}\|,\|\oB_{2}\|,\|\cA^{-1}\|)$
such that 
\begin{align}
\|\left(\widehat{\cT}(\tau)-\e^{-\tau\cC}\right)\cA^{-1}\| & =\|\int_{0}^{\tau}\left\{ \widehat{\cT}'(\sigma)+\widehat{\cT}(\sigma)\cC\right\} \e^{-(\tau-\sigma)\cC}\cA^{-1}\d\sigma\|\nonumber \\
 & =\|\int_{0}^{\tau}\left\{ \widehat{\cT}'(\sigma)+\widehat{\cT}(\sigma)\cC\right\} \cA^{-1}\cA\cC^{-1}\e^{-(\tau-\sigma)\cC}\cC\cA^{-1}\d\sigma\|\nonumber \\
 & \leq\int_{0}^{\tau}\|\left\{ \widehat{\cT}'(\sigma)+\widehat{\cT}(\sigma)\cC\right\} \cA^{-1}\|\cdot\|\cA\cC^{-1}\|\cdot\|\e^{-(\tau-\sigma)\cC}\|\cdot\|\cC\cA^{-1}\|\d\sigma\nonumber \\
 & \leq C\int_{0}^{\tau}\sigma\e^{(\tau-\sigma)\|\cB\|}\d\sigma\leq\frac{C}{2}\tau^{2}\e^{\tau\|\cB\|},\label{eq:EstimateInLemmaO(tau2)}
\end{align}
which proves the claim.
\end{proof}

\subsection{Convergence result for holomorphic semigroups}

We are now able to state and prove the main theorem, which show convergence
in operator norm with convergence rate estimate of $O(\frac{\log n}{n})$.
\begin{thm}
\label{thm:MainConvergenceResultHolomorphicSemigroups}Let $-\oA_{j}:\dom\left(\oA_{j}\right)\subset X_{j}\to X_{j}$
be generators of holomorphic contraction semigroups, and let $\oB_{1}:X_{2}\to X_{1}$,
$\oB_{2}:X_{1}\to X_{2}$ be bounded. Then, there are constants $C,\eta>0$
such that for all $t\geq0$ and $n\in\N$, we have 
\[
\|\cT(t/n)^{n}-\e^{-t\cC}\|\leq\frac{C}{n}\e^{t\eta}\e^{4t\left(\|\oB_{1}\|+\|\oB_{2}\|\right)}\left(\log n+t^{2}\right).
\]
\end{thm}

\begin{proof}
It is clear that nothing has to be shown for $t=0$ or $n=1$. So
let $t>0$ and $n\geq2$, and let us introduce $\tau=t/n$. If $-\cA$
and $-\cC$ are not boundedly invertible then, by introducing a shift
$\eta>0$ they can be made invertible and the previous estimates remain
unchanged. Indeed, defining $\widetilde{\oA}_{j}=\oA_{j}+\eta$ and
$\widetilde{\cC}=\cC+\eta$ for $\eta>0$ such that $\widetilde{\oA_{j}},\widetilde{\cC}$
are invertible, we observe that $\widetilde{\oE}_{j}(\tau):=\e^{-\tau\widetilde{\oA}_{j}}=\e^{-\tau\eta}\oE_{j}(\tau)$.
Clearly, $-\widetilde{\cC}$ is a generator of a holomorphic semigroup
and we have
\begin{equation}
\|t\widetilde{\cC}\e^{-t\widetilde{\cC}}\|=\|t\left(\cC+\eta\right)\e^{-t\cC}\e^{-t\eta}\|\leq M_{\cC}+t\eta\e^{-t\eta}\|\e^{-t\cC}\|\leq M_{\cC}+\frac{1}{2}\e^{t\left(\|\oB_{1}\|+\|\oB_{2}\|\right)},\label{eq:EstimateShiftedHolomorphicSemigroup}
\end{equation}
where we have used that $x\e^{-x}\leq\tfrac{1}{2}$ for all $x\geq0$.

Moreover, we define $\widetilde{\oX}_{1}(\tau):=\e^{-\tau\eta}\oX_{1}(\tau).$
Then, we have for all $\tau\geq0$ that
\[
\e^{\tau\eta}\widetilde{\cT}(\tau):=\e^{\tau\eta}\begin{pmatrix}\widetilde{\oE}_{1}(\tau) & \widetilde{\oX}_{1}(\tau)\\
\oX_{2}(\tau)\widetilde{\oE}_{1}(\tau) & \oX_{2}(\tau)\widetilde{\oX}_{1}(\tau)+\widetilde{\oE}_{2}(\tau)
\end{pmatrix}=\cT(\tau).
\]
Then, for all $\tau\geq0$ we have that 
\[
\|\widetilde{\oX}_{1}(\tau)\|\leq\tau\|\oB_{1}\|,\,\|\widetilde{\oA}_{1}\widetilde{\oX}_{1}(\tau)\|\leq\|\oA_{1}\oX_{1}(\tau)\|+\eta\|\oX_{1}(\tau)\|,\,\|\widetilde{\oA}_{2}\oX_{2}(\tau)\|\leq\|\oA_{2}\oX_{2}(\tau)\|+\eta\|\oX_{2}(\tau)\|.
\]
In particular, Lemma \ref{lem:AuxLemmaHolomorphicSemigroups} and
Lemma \ref{lem:AuxLemmaO(tau)Estimate} can now be adapted to the
shifted situation. Moreover, we have $\widetilde{\oX}_{1}'(\tau)-\oX_{1}'(\tau)=-\eta\e^{-\tau\eta}\oX_{1}(\tau)$,
which shows that the estimate \eqref{eq:EstimateInLemmaO(tau2)} holds,
so Lemma \ref{lem:AuxLemmaO(tau2)Estimate} can also be adapted to
the shifted situation. We have
\begin{align*}
\e^{-t\widetilde{\cC}}-\widetilde{\cT}(t/n)^{n} & =\left(\e^{-\tau\widetilde{\cC}}\right)^{n}-\widetilde{\cT}(\tau)^{n}\\
 & =\sum_{k=0}^{n-1}\e^{-\tau(n-1-k)\widetilde{\cC}}\left(\e^{-\tau\widetilde{\cC}}-\widetilde{\cT}(\tau)\right)\widetilde{\cT}(\tau)^{k}\\
 & =\e^{-\tau(n-1)\widetilde{\cC}}\widetilde{\cC}\widetilde{\cC}^{-1}\widetilde{\cA}\widetilde{\cA}^{-1}\left(\e^{-\tau\widetilde{\cC}}-\widetilde{\cT}(\tau)\right)+\left(\e^{-\tau\widetilde{\cC}}-\widetilde{\cT}(\tau)\right)\widetilde{\cA}^{-1}\widetilde{\cA}\widetilde{\cT}(\tau)^{n-1}\\
 & \qquad+\sum_{k=1}^{n-2}\e^{-\tau(n-1-k)\widetilde{\cC}}\left(\e^{-\tau\widetilde{\cC}}-\widetilde{\cT}(\tau)\right)\widetilde{\cA}^{-1}\widetilde{\cA}\widetilde{\cT}(\tau)^{k},
\end{align*}
where we have used the product $\mathbb{\cS}^{k}-\mathbb{\cT}^{k}=\sum_{j=0}^{k-1}\mathbb{\cS}^{k-1-j}(\mathbb{\cS}-\cT)\mathbb{\cT}^{j}$.
Then, by Lemma \ref{lem:AuxLemmaHolomorphicSemigroups}, \ref{lem:AuxLemmaO(tau)Estimate}
and \ref{lem:AuxLemmaO(tau2)Estimate} , we get
\begin{align*}
\|\e^{-t\widetilde{\cC}} & -\widetilde{\cT}(t/n)^{n}\|\\
 & \leq\|\e^{-\tau(n-1)\widetilde{\cC}}\widetilde{\cC}\|\cdot\|\widetilde{\cC}^{-1}\widetilde{\cA}\|\cdot\|\widetilde{\cA}^{-1}\left(\e^{-\tau\widetilde{\cC}}-\widetilde{\cT}(\tau)\right)\|+\|\left(\e^{-\tau\widetilde{\cC}}-\widetilde{\cT}(\tau)\right)\widetilde{\cA}^{-1}\|\cdot\|\widetilde{\cA}\widetilde{\cT}(\tau)^{n-1}\|\\
 & \qquad+\sum_{k=1}^{n-2}\|\e^{-\tau(n-1-k)\widetilde{\cC}}\|\cdot\|\left(\e^{-\tau\widetilde{\cC}}-\widetilde{\cT}(\tau)\right)\widetilde{\cA}^{-1}\|\cdot\|\widetilde{\cA}\widetilde{\cT}(\tau)^{k}\|\\
 & \leq\frac{1}{\tau(n-1)}\left(M_{\cC}+\frac{1}{2}\e^{\tau(n-1)\|\cB\|}\right)\|\widetilde{\cC}^{-1}\widetilde{\cA}\|\cdot C_{2}\left(1+\e^{\tau\|\cB\|}\right)\tau+\\
 & \qquad+C_{3}\tau^{2}\left\{ C_{1}\e^{t\|\cB\|}(1+\log(n-1))+\frac{M_{A}}{(n-1)\tau}\right\} +\\
 & \qquad+\sum_{k=1}^{n-2}\e^{\tau(n-1-k)\|\cB\|}\e^{-t\eta}C_{3}\tau^{2}\left\{ C_{1}\e^{t\|\cB\|}(1+\log k)+\frac{M_{A}}{k\tau}\right\} \\
 & \leq\frac{1}{(n-1)}\left(M_{1}+M_{2}\e^{t\left(\|\oB_{1}\|+\|\oB_{2}\|\right)}\right)+C_{3}\frac{t}{n}\left\{ C_{1}t\e^{t\|\cB\|}+\frac{M_{A}}{(n-1)}\right\} +\\
 & \qquad+\e^{t\|\cB\|}\e^{-t\eta}C_{3}\frac{t}{n}\left\{ C_{1}\e^{t\|\cB\|}n(1+\log n)\frac{t}{n}+M_{A}\log n\right\} \\
 & \leq\frac{C}{n}\e^{2t\|\cB\|}\left(\log n+t^{2}\right),
\end{align*}
for a constant $C>0$, where we have used that 
\[
\left(1+\log(n-1)\right)\frac{1}{n}\leq1,\quad\sum_{k=1}^{n-2}\frac{1}{k}\leq\log n,\quad\sum_{k=1}^{n-2}\log k\leq n\log n.
\]
For the operators without the shift this means
\[
\|\cT(t/n)^{n}-\e^{-t\cC}\|\leq\|\left(\e^{t/n\eta}\widetilde{\cT}(\frac{t}{n})\right)^{n}-\e^{t\eta}\e^{-t\widetilde{\cC}}\|=\e^{t\eta}\|\e^{-t\widetilde{\cC}}-\widetilde{\cT}(t/n)^{n}\|\leq\frac{C}{n}\e^{t\eta}\e^{2t\|\cB\|}\left(\log n+t^{2}\right),
\]
which shows the claimed estimate.
\end{proof}

\subsection{Operator-norm convergence in weaker norm}

Interestingly, we get on the subspace $\dom\left(\cA\right)\subset X$
a similar convergence result as Theorem\ref{thm:MainConvergenceResultHolomorphicSemigroups}
without assuming that the semigroups $\e^{-t\oA_{j}}$ are holomorphic.
For this we assume that $\oA_{j}$ are boundedly invertible (as we
have seen in the proof of Theorem \ref{thm:MainConvergenceResultHolomorphicSemigroups}
we could otherwise introduce a shift). We define a new operator norm
for bounded operators $\cB:X\to X$, 
\[
\|\cB\|_{\cA}:=\sup_{f\in\dom\left(\cA\right):\|\cA f\|\leq1}\|\cB f\|=\sup_{g\in X:\|g\|\leq1}\|\cB\cA^{-1}g\|=\|\cB\cA^{-1}\|.
\]
If$-\oA_{j}$ is unbounded, a bound on $\|\cB\|_{\cA}$ does not provide
a bound on $\|\cB\|$ in general. The crucial observation is Lemma
\ref{lem:AuxLemmaO(tau2)Estimate}, which provides an bound $\|\cT(\tau)-\e^{-\tau\cC}\|_{\cA}=O(\tau^{2})$.
Note that Lemma \ref{lem:AuxLemmaO(tau2)Estimate} here is a better
estimate than the analogous results in \cite{NeStZa17CREASO,NeStZa18ONCTPF,NeStZa18RONCTPF,NeStZa19TPFLEEHS,NeStZa20CRETPA}\textcolor{blue}{{}
}because the spatial regularization is only needed once to obtain
an estimate of order $O(\tau)$. We refer also to \cite{JahLub00EBEOS,HanOst09ESUO}
for comparable results related to the Trotter-product formula.
\begin{thm}
\label{thm:ConvergenceResultWeakOperatorNorm}Let $-\oA_{j}:\dom\left(\oA_{j}\right)\subset X_{j}\to X_{j}$
be generators of contraction semigroups, and let $\oB_{1}:X_{2}\to X_{1}$,
$\oB_{2}:X_{1}\to X_{2}$ be bounded. Moreover let $\cA$ and $\cC$
be boundedly invertible. Then, there is a constant $C=C(\|\oB_{2}\|,\|\oB_{1}\|,\|\oA_{1}^{-1}\|,\|\oA_{2}^{-1}\|)>0$
such that for all $t>0$ and $n\geq1,$we have
\[
\|\cT(t/n)^{n}-\e^{-t\cC}\|_{\cA}\leq\frac{C}{n}t^{2}\e^{2t(\|\oB_{1}\|+\|\oB_{2}\|)}\ .
\]
\end{thm}

\begin{proof}
There is nothing to show for $t=0$ and $n=1$. So let $t>0$ and
$n\geq2$. Introducing, $\tau=\frac{t}{n}$ and using the product
$\mathbb{\cS}^{k}-\mathbb{\cT}^{k}=\sum_{j=0}^{k-1}\mathbb{\cS}^{k-1-j}(\mathbb{\cS}-\cT)\mathbb{\cT}^{j}$,
we have
\begin{align*}
\left(\cT(\tau)^{n}-\e^{-t\cC}\right)\cA^{-1} & =\left(\cT(\tau)^{n}-\left(\e^{-\tau\cC}\right)^{n}\right)\cA^{-1}\\
 & =\sum_{k=0}^{n-1}\cT(\tau)^{n-k-1}\left(\cT(\tau)-\e^{-\tau\cC}\right)\e^{-\tau k\cC}\cA^{-1}\\
 & =\cT(\tau)^{n-1}\left(\cT(\tau)-\e^{-\tau\cC}\right)\cA^{-1}+\left(\cT(\tau)-\e^{-\tau\cC}\right)\cA^{-1}\cA\cC^{-1}\e^{-\tau(n-1)\cC}\cC\cA^{-1}\\
 & \qquad+\sum_{k=1}^{n-2}\cT(\tau)^{n-k-1}\left(\cT(\tau)-\e^{-\tau\cC}\right)\cA^{-1}\cA\cC^{-1}\e^{-\tau k\cC}\cC\cA^{-1}.
\end{align*}
By Lemma \ref{lem:AuxLemmaO(tau2)Estimate}, the first and the second
term can be estimated by $C\e^{t(\|\oB_{1}\|+\|\oB_{2}\|)}\tau^{2}$.
For the sum in the last term, we have the bound $Cn\e^{t(\|\oB_{1}\|+\|\oB_{2}\|)}\tau^{2}$.
Hence, we conclude
\[
\|\left(\cT(\tau)^{n}-\e^{-t\cC}\right)\cA^{-1}\|\leq\frac{C}{n}t^{2}\e^{t(\|\oB_{1}\|+\|\oB_{2}\|)}\,.
\]
\end{proof}

\subsection{Other similar approximations\label{subsec:Other-approximations}}

Analogue convergence results as \ref{thm:MainConvergenceResultHolomorphicSemigroups}
and \ref{thm:ConvergenceResultWeakOperatorNorm} can also be shown
for other approximations.

\subsubsection{Transposed approximation}

Instead on applying first $\cT_{1}$ and then $\cT_{2}$, one could
also consider the transposed approximation $\cT_{\mathrm{T}}(\tau):=\cT_{1}(\tau)\cT_{2}(\tau)$.
Since, we have

\[
\cT_{\mathrm{T}}-\cT=\cT_{1}\cT_{2}-\cT_{2}\cT_{1}=\begin{pmatrix}\oX_{1}\oX_{2} & \oX_{1}(\oE_{2}-\oI)\\
\oX_{2}(\oI-\oE_{1}) & \oX_{2}\oX_{1}
\end{pmatrix},
\]
we immediately obtain that $\left(\cT_{\mathrm{T}}(\tau)-\cT(\tau)\right)\cA^{-1}=O(\tau^{2})$,
which provides an analogue of Lemma \ref{lem:AuxLemmaO(tau)Estimate}.
Hence, we get for $\cT_{\mathrm{T}}(t/n)^{n}-\e^{-t\cC}$ the same
operator-norm convergence result as in Theorem \ref{thm:MainConvergenceResultHolomorphicSemigroups}.

\subsubsection{Symmetrized approximation}

In the proof of Lemma \ref{lem:AuxLemmaO(tau)Estimate}, we had already
used the symmetric approximation $\widehat{\cT}$, which has the simple
form
\[
\widehat{\cT}(t)=\begin{pmatrix}\oE_{1}(t) & \oX_{1}(t)\\
\oX_{2}(t) & \oE_{2}(t)
\end{pmatrix}=\e^{-t\cA}+\begin{pmatrix}\cdot & \int_{0}^{t}\e^{-s\oA_{1}}\d s\\
\int_{0}^{t}\e^{-s\oA_{2}}\d s & \cdot
\end{pmatrix}\circ\cB\,.
\]
In Lemma \ref{lem:AuxLemmaO(tau)Estimate} it is shown that we have
an analogue result of the form $\left(\widehat{\cT}(\tau)-\e^{-t\cC}\right)\cA^{-1}=O(\tau^{2})$
holds. Hence, we get for $\widehat{\cT}(t/n)^{n}-\e^{-t\cC}$ the
same operator-norm convergence result as in Theorem \ref{thm:MainConvergenceResultHolomorphicSemigroups}.

\subsubsection{Naive solution of the integral}

In the convergence result in the strong topology Proposition (\ref{prop:StabilityStrongConvergence}),
we have already discussed the approximation, where the integral in
the coupling term is naively solved and $\oX_{j}(\tau)$ is replaced
by $\tau\oB_{j}$, which leads to $\cT_{\mathrm{B}}=\begin{pmatrix}\oE_{1} & \tau\oB_{1}\\
\tau\oB_{2} & \oE_{2}
\end{pmatrix}$. However, an analogue convergence result like $\left(\cT_{\mathrm{B}}(\tau)-\e^{-\tau\cC}\right)\cA^{-1}=O(\tau)$
is not clear.

\section{Remarks on unbounded coupling $\cB$\label{sec:UnboundedCoupling}}

In this section, we briefly comment on the situation where the coupling
between the spaces $X_{1}$ and $X_{2}$ is given by an unbounded
linear operator $\cB=\begin{pmatrix} & \oB_{1}\\
\oB_{2}
\end{pmatrix}$, $\oB_{1}:\dom(\oB_{1})\subset X_{2}\to X_{1}$ and $\oB_{2}:\dom(\oB_{2})\subset X_{1}\to X_{2}$.
Throughout the section, we assume that $-\oA_{j}$ are boundedly invertible
generators of holomorphic contraction semigroups.

\subsection{Existence of solution operators for the perturbed system and the
inhomogeneous abstract Cauchy problem}

For generators of holomorphic semigroups $-\oA_{j}$ it is possible
to define fractional powers $\oA_{j}^{\alpha}$, $\alpha\in[0,1]$
interpolating between $\oA_{j}$ and $\oI$. In previous similar works
(see e.g. \cite{CacZag01ONCTPFHS,NeStZa20CRETPA}), it is assumed
that there is an $\alpha\in[0,1[$ such that $\dom(\cA^{\alpha})\subset\dom(\cB)$
and that $\cB\cA^{-\alpha}:X\to X$ is bounded, or equivalently that
$\oB_{1}\oA_{2}^{-\alpha}:X_{2}\to X_{1}$ and that $\oB_{2}\oA_{1}^{-\alpha}:X_{1}\to X_{2}$
are bounded. Then, $\cB$ is $\cA$-bounded with relative bound zero,
and hence, the sum $-\cA+\cB$ is a generator of a holomorphic semigroup,
by classical perturbation results \cite[Theorem III.2.10]{EngNag00OPSLEE}.

Moreover, we are going to assume that there is a $\beta\in[0,1[$
such that $\oA_{i}^{-\beta}\oB_{i}$ is bounded, or equivalently,
that $\cA^{-\beta}\cB$ is bounded. Then for all $\tau\geq0$ the
split-step approximation operators $\cT_{i}(\tau):X\to X$ defined
by \eqref{eq:Definition=00005CcT} are bounded. Indeed, we have for
all $y\in X_{2}$ that
\begin{align*}
\|\int_{0}^{\tau}\d\sigma\e^{-\sigma\oA_{1}}\oB_{1}y\| & =\|\int_{0}^{\tau}\d\sigma\e^{-\sigma\oA_{1}}\oA_{1}^{\beta}\oA_{1}^{-\beta}\oB_{1}y\|\leq\int_{0}^{\tau}\|\e^{-\sigma\oA_{1}}\oA_{1}^{\beta}\|\d\sigma\cdot\|\oA_{1}^{-\beta}\oB_{1}y\|\\
 & \leq C_{\beta}\int_{0}^{\tau}\sigma^{-\beta}\d\sigma\cdot\|\oA_{1}^{-\beta}\oB_{1}\|\cdot\|y\|\leq C_{\beta}\tau^{1-\beta}\|\oA_{1}^{-\beta}\oB_{1}\|\cdot\|y\|,
\end{align*}
where we have used, that for generators $-\oA$ of bounded holomorphic
semigroups, there is a constant $C_{\beta}>0$ such that for all $t>0$
we have the estimate
\[
\|\oA^{\beta}\e^{-t\oA}\|\leq\frac{C_{\beta}}{t^{\beta}}\,.
\]
Hence, we get that $\cT_{1}(\tau):X\to X$ is bounded. Similarly we
get that also $\cT_{2}(\tau):X\to X$ is bounded, and thus defining
a bounded time-discretization $\cT(\tau)=\cT_{2}(\tau)\cT_{1}(\tau)$,
which satisfies $\cT(0)=\cI$.

\subsection{Stability of the approximation family}

To ensure that $\left\{ \cT(\tau)\right\} _{\tau\geq0}$ is a reasonable
approximation family, we have to show that $\cT$ is stable, i.e.
the family $\cT(t/n)^{n}$ is uniformly bounded. We shortly discussed
why stability is delicate and in general cannot expected under the
assumptions here.

We have seen that the coupling terms in $\cT_{j}$ are bounded by
$O(\tau^{1-\beta})$. So we get (neglecting bounded operators)
\begin{align*}
\cT(\tau)=\cT_{2}(\tau)\cT_{1}(\tau) & =\begin{pmatrix}\oI & \cdot\\
\int_{0}^{\tau}\d\sigma\e^{-\sigma\oA_{2}}\oB_{2} & \e^{-\tau\oA_{2}}
\end{pmatrix}\begin{pmatrix}\e^{-\tau\oA_{1}} & \int_{0}^{\tau}\d\sigma\e^{-\sigma\oA_{1}}\oB_{1}\\
\cdot & \oI
\end{pmatrix}\\
 & \thickapprox\begin{pmatrix}\oI & \cdot\\
\tau^{1-\beta}\oI & \oI
\end{pmatrix}\begin{pmatrix}\oI & \tau^{1-\beta}\oI\\
\cdot & \oI
\end{pmatrix}=\begin{pmatrix}\oI & \tau^{1-\beta}\\
\tau^{1-\beta} & \left(\tau^{2-2\beta}+1\right)\oI
\end{pmatrix}.
\end{align*}
If $\beta=0$, then we have $\begin{pmatrix}1 & \tau\\
\tau & 1+\tau^{2}
\end{pmatrix}=\begin{pmatrix}1 & \cdot\\
\cdot & 1
\end{pmatrix}+\tau\begin{pmatrix}\cdot & 1\\
1 & \tau
\end{pmatrix},$which has bounded powers expressed by the matrix exponential.

However, we have that on $\R^{2}$ that the matrix powers of $\begin{pmatrix}1 & \tau^{1-\beta}\\
\tau^{1-\beta} & 1+\tau^{2-2\beta}
\end{pmatrix}=:\mathbf{P}(\tau^{1-\beta})$ for $\beta\in]0,1]$ are unbounded. Indeed, fixing $x=\tau^{1-\beta}$
we have that $v(x)=\left(\frac{1}{2}(-x+\sqrt{4+x^{2}}),1\right)^{\mathrm{T}}$
is an eigenvector of $\mathbf{P}(x)$, with $\mathbf{P}(x)v(x)=\left(1+\frac{x}{2}\left(x+\sqrt{4+x^{2}}\right)\right)v(x)$,
and $v(x)$ is uniformly bounded for $x=\tau^{1-\beta}$ as $\tau\to0$,
with $v(x)\to(1,1)^{\mathrm{T}}$. Hence, 
\begin{align*}
\|\cT(t/n)^{n}\| & \geq\frac{1}{\|v(x)\|}\|\begin{pmatrix}1 & x\\
x & 1+x^{2}
\end{pmatrix}^{n}v(x)\|=\frac{1}{\|v(x)\|}\|\left(1+\frac{x}{2}\left(x+\sqrt{4+x^{2}}\right)\right)^{n}v(x)\|.\\
 & \geq\left(1+\frac{x}{2}\left(x+\sqrt{4+x^{2}}\right)\right)^{n}\geq1+n\frac{x}{2}\left(x+\sqrt{4+x^{2}}\right)=nx=n^{\beta}t^{1-\beta},
\end{align*}
which tends to infinity as $n\to\infty$. This means, stability of
${\cal T}(\tau)$ is in general not clear, and, hence, convergence
of $\cT(\tau)^{n}$ to $\e^{-t\cC}$ (even in the strong topology)
cannot be expected.

{\footnotesize{}\bibliographystyle{/Users/astephan/Documents/Science/my_alpha}
\bibliography{/Users/astephan/Documents/Science/total-bib-file}
}{\footnotesize\par}
\end{document}